\documentclass[11pt,a4paper,twoside]{article}
\usepackage{amsthm,amsfonts,amsmath,amscd,amssymb}
\usepackage{latexsym}
\usepackage{euscript}
\usepackage{enumitem}
\usepackage{graphicx}
\usepackage{tikz}
\usepackage{caption}
\usepackage{subcaption}
\usepackage{cite}
\usepackage[utf8]{inputenc}
\usepackage[english]{babel}
\usepackage{xcolor}

\usepackage{jmpag} 

\begin{document}



\title{Novel view on classical convexity theory}

\author{Vitali Milman}
\address{Tel-Aviv University, Tel-Aviv, 69978, Israel}
\email{milman@tauex.tau.ac.il}

\author{Liran Rotem}

\address{Technion – Israel Institute of Technology,  Haifa, 32000, Israel}
\email{lrotem@technion.ac.il}


\BeginPaper 


\theoremstyle{jmpagplain}
\newtheorem{fact}[theorem]{Fact}
\newtheorem{problem}[theorem]{Problem}

\newcommand\RR{\mathbb{R}}
\newcommand\pphi{\varphi}
\newcommand\dd{\mathrm{d}}
\newcommand\BB{\mathcal{B}}
\newcommand\FF{\mathcal{F}}
\newcommand\GG{\mathcal{G}}
\newcommand\KK{\mathcal{K}}
\newcommand\cof{\text{co\ensuremath{\pphi}}}
\newcommand\flower{\text{\ensuremath{\clubsuit}}}
\newcommand\conv{\operatorname{conv}}
\newcommand\outc{\operatorname{out}}
\newcommand\inc{\operatorname{in}}
\newcommand\cone{\operatorname{cone}}



\begin{abstract}
Let $B_{x}\subseteq\RR^{n}$ denote the Euclidean ball with diameter
$[0,x]$, i.e. with with center at $\frac{x}{2}$ and radius $\frac{\left|x\right|}{2}$.
We call such a ball a petal. A flower $F$ is any union of petals,
i.e. $F=\bigcup_{x\in A}B_{x}$ for any set $A\subseteq\RR^{n}$.
We showed earlier in \cite{Milman2019} that the family of all flowers
$\FF$ is in 1-1 correspondence with $\KK_{0}$ \textendash{} the
family of all convex bodies containing $0$. Actually, there are two
essentially different such correspondences. We demonstrate a number
of different non-linear constructions on $\FF$ and $\KK_{0}$. Towards
this goal we further develop the theory of flowers.

\key{convex bodies, flowers, spherical inversion, duality, powers, Dvoretzky's theorem}

\end{abstract}



\section{Introduction: Flowers}

We start with the Euclidean unit ball $B_{2}^{n}\subseteq\RR^{n}$.
We denote by $B(x,r)$ the Euclidean ball centered at $x$ and has
radius $r>0$. Let $\BB_{0}$ be the family of all balls which contain
$0$. In other words $B(x,r)\in\BB_{0}$ if and only if $\left|x\right|\le r$,
where $\left|x\right|$ is the Euclidean norm of $x$. Also, we write
$B_{x}=B\left(\frac{x}{2},\frac{\left|x\right|}{2}\right)$, i.e.
the ball that has $[0,x]$ as its diameter. We also write $B_{2}^{n}=B(0,1)$
for the unit ball. 
\begin{definition}
A \emph{flower} is any set of the form $F=\bigcup_{\alpha}B_{\alpha}$
for a collection of balls $\left\{ B_{\alpha}\right\} _{\alpha}\subseteq\BB_{0}$.
We denote the family of all flowers by $\FF$. 
\end{definition}

We state that every flower $F$ uniquely represents a pair $\left(K,K^{\circ}\right)$
where $K\in\KK_{0}$, i.e. a closed convex set containing $0$, and
$K^{\circ}$ is the canonical dual of $K$. More precisely, we call
$K$ the core of $F$ if 
\[
K=\left\{ x\in\RR^{n}:\ B_{x}\subseteq F\right\} .
\]
 Let $\phi$ be spherical inversion, i.e. $\phi(x)=\frac{x}{\left|x\right|^{2}}$
for $x\ne0$. For any star body $A$ (i.e. such that $\lambda A\subseteq A$
for all $0\le\lambda\le1$), define the co-image of $\phi$ by 
\[
\cof\left(\text{A}\right)=\overline{\left\{ x:\ x\notin\phi(A)\right\} }
\]
 (i.e. the closure of the complement $\phi(A)^{c}$). Note that the
closure is always radial. Then consider the set $T:=\cof\left(F\right)$.
\begin{fact}[\cite{Milman2019}]
For any flower $F$, the bodies $K$ and $T$ from the above construction
belong to $\KK_{0}$, and $T=K^{\circ}$. 
\end{fact}

Every $K\in\KK_{0}$ it the core of a unique flower which we denote
by $F=K^{\flower}$, or sometimes by $F=\flower K$. The map $\flower:\KK_{0}\to\FF$
is called the flower map, and we denote its inverse (the core operation)
by $K=F^{-\flower}$. We therefore have one to one and onto maps
\[
\KK_{0}\xrightarrow{\flower}\FF\xrightarrow{\cof}\KK_{0},
\]
 and their composition is exactly the duality map: $\cof\left(K^{\flower}\right)=K^{\circ}$.
So, every flower $F$ ``sees'' simultaneously a convex body $K=F^{-\flower}$
and its dual $K^{\circ}=\cof\left(F\right)$. Since $\cof$ is an
involution, we obtain an equivalent definition of the class of flowers
$\FF$: we simply have $\FF=\cof\left(\KK_{0}\right)$, i.e. \textbf{flowers
are the complements of inversions of convex bodies} (containing the
origin). 

Also, these maps are uniquely defined by their order reversing/preserving
properties:
\begin{proposition}
\begin{enumerate}
\item Let $f:\FF\to\KK_{0}$ be a one to one and onto map such that $f$
and $f^{-1}$ preserve the order of inclusion (i.e. $F_{1}\subseteq F_{2}$
if and only if $f(F_{1})\subseteq f\left(F_{2}\right)$). Then there
exists an invertible linear map $u:\RR^{n}\to\RR^{n}$ such that $f(F)=u\left(F^{-\flower}\right)$. 
\item Let $g:\FF\to\KK_{0}$ be a one to one and onto map such that $g$
and $g^{-1}$ reverse the order of inclusion (i.e. $F_{1}\subseteq F_{2}$
if and only if $g(F_{1})\supseteq g\left(F_{2}\right)$). Then there
exists an invertible linear map $u:\RR^{n}\to\RR^{n}$ such that $g(F)=u\left(\cof(F)\right)$. 
\end{enumerate}
\end{proposition}

\begin{proof}
For $(1)$, define $h:\KK_{0}\to\KK_{0}$ by $h(K)=f\left(K^{\flower}\right)$.
Then $h$ is a bijection and $h$, $h^{-1}$ preserve order. By a
theorem of \cite{Artstein-Avidan2008} it follows that $h\left(K\right)=u\left(K\right)$
for some invertible linear map $u:\RR^{n}\to\RR^{n}$ (Technically
Theorem 10 of \cite{Artstein-Avidan2008} assumes that $h$ is an
order-reversing \emph{involution}, but the proof works for our situation
as well. For a proof of the result as we use it see e.g. Theorem 2
of \cite{Slomka2011}, which proves a stronger statement and gives
references to other related works). Then 
\[
f(F)=h\left(F^{-\flower}\right)=u\left(F^{-\flower}\right)
\]
 as we wanted.

For $(2)$, note that $\cof\left(\cof\left(A\right)\right)=A$ for
all (radially closed) star bodies $A$. In particular the inverse
of the map $\cof:\FF\to\KK_{0}$ is also $\cof:\KK_{0}\to\FF$. If
we now define $h:\KK_{0}\to\KK_{0}$ by $h(K)=g\left(\cof\left(K\right)\right)$,
then again $h$ and $h^{-1}$ are order preserving bijections so $h(K)=u\left(K\right)$.
Hence 
\[
g(F)=h\left(\cof\left(F\right)\right)=u\left(\cof\left(F\right)\right).
\]
 
\end{proof}
The definitions of the flower as given above are equivalent to the
following third definition: \textbf{A flower is any set of the form
$F=\bigcup_{\alpha}B_{x_{\alpha}}$ for any set $\left\{ x_{\alpha}\right\} _{\alpha}\subseteq\RR^{n}$}.
Daniel Hug informed us that this definition was previously used in
the study of Voronoi tessellations, where it is sometimes called the
Voronoi flower of a convex body. The equivalence of these definitions
is, of course, a statement which should be proved (see \cite{Milman2019}).
It follows from the fact that every ball $B\in\BB_{0}$ is a flower
in this new sense. 

In the same paper we also present a fourth (equivalent) definition:
Let $h_{K}:S^{n-1}\to[0,\infty]$ be the supporting functional of
a convex body $K\in\KK_{0}$. Note that we consider $h_{K}$ as a
function only on the sphere $S^{n-1}=\left\{ x\in\RR^{n}:\ \left|x\right|=1\right\} $
and not as a 1-homogeneous function on $\RR^{n}$. Let $F$ be the
star body with radial function $r_{F}(\theta)=h_{K}(\theta)$ for
all $\theta\in S^{n-1}$. Then $F$ is a flower and $F=K^{\flower}$.
As the converse is also true,\textbf{ flowers are exactly the star
bodies whose radial function is convex }(as a function on the sphere,
meaning its $1$-homogeneous extension is convex on $\RR^{n}$). 

This last description of the flower map $\flower$ is very useful
in different computations and constructions we will describe. It was
actually the original definition given in \cite{Milman2019}. 

In \cite{Milman2019} we also introduced flower mixed volumes. Consider
any collection of flowers $\left\{ F_{i}\right\} _{i=1}^{m}$ in $\RR^{n}$
and non-negative integers $\left\{ \lambda_{i}\right\} _{i=1}^{m}$,
and construct a new flower by 
\[
\GG=\GG\left(\left\{ F_{i}\right\} _{i},\left\{ \lambda_{i}\right\} _{i}\right)=\left(\sum_{i=1}^{m}\lambda_{i}F_{i}^{-\flower}\right)^{\flower}.
\]
 Then $\left|\GG\right|$, the volume of $\GG$, is a homogeneous
polynomial of degree $n$: 
\[
\left|\GG\right|=\sum_{1\le i_{1},i_{2},\ldots,i_{n}\le m}V(F_{i_{1}},F_{i_{2}},\ldots,F_{i_{n}})\lambda_{i_{1}}\lambda_{i_{2}}\cdots\lambda_{i_{n}},
\]
 where as usual we take the coefficients $V(F_{i_{1}},F_{i_{2}},\ldots,F_{i_{n}})$
to be invariant with respect to permutations of their arguments. We
call these coefficients flower mixed volumes. For the cores $K_{i}=F_{i}^{-\flower}$
we also set 
\[
V_{\flower}(K_{1},K_{2},\ldots,K_{n})=V(F_{1},F_{2},\ldots,F_{n}),
\]
 and an explicit formula for these numbers is (see \cite{Milman2019})
\[
V_{\flower}(K_{1},K_{2},\ldots,K_{n})=\left|B_{2}^{n}\right|\cdot\int_{S^{n-1}}\prod_{i=1}^{n}h_{K_{i}}(\theta)\mathrm{d}\sigma(\theta).
\]

Finally, let us mention a few more facts from \cite{Milman2019} about
flowers: If $F=\bigcup_{x\in A\subseteq\RR^{n}}B_{x}$ is a flower,
then necessarily $F^{-\flower}=\conv A$. If $F_{1}$ and $F_{2}$
are flowers, so are both the radial sum $F_{1}\widetilde{+}F_{2}$
(defined by $r_{F_{1}\widetilde{+}F_{2}}=r_{F_{1}}+r_{F_{2}})$ and
the Minkowski sum $F_{1}+F_{2}$. Note that we are taking the Minkowski
sum of not-necessarily-convex sets. Also, if $F$ is a flower and
$E\subseteq\RR^{n}$ is any linear subspace, then $F\cap E$ is also
a flower, and in fact 
\[
P_{E}\left(F^{-\flower}\right)=\left(F\cap E\right)^{-\flower},
\]
 where $P_{E}$ denotes the orthogonal projection onto $E$. 

Moreover $P_{E}F$ is also a flower, even though we do not have an
independent description of $\left(P_{E}F\right)^{-\flower}$. Since
$P_{E}F\supseteq F\cap E$ we know that $\left(P_{E}F\right)^{-\flower}\supseteq P_{E}\left(F^{-\flower}\right)$,
but we do not have a good understanding of this set. 

Finally, we mention that if $F$ is a flower so is its convex hull
$\conv F$. In this case there is a description in \cite{Milman2019}
of $\left(\conv F\right)^{-\flower}$ in terms of $F^{-\flower}$
and the so-called reciprocity map, but we will not explain it further
here. 

\section{Non-linear constructions using flowers}

To avoid uninteresting technicalities, let us assume from this point
on that our flowers are always compact and contain the origin at their
interior. Recall from the introduction that if $F=\bigcup_{x\in A\subseteq\RR^{n}}B_{x}$
is such a flower then the core $K=F^{-\flower}$ satisfies $K=\conv A$.
However, this condition does not define $A$ uniquely, so there are
many representations of the same flower $F$ by different sets $A$.
We would like to select one canonical representation. We call $F=\bigcup_{x\in A}B_{x}$
canonical if we have $A=\partial K$, where $\partial K$ denotes
the boundary of $K$. This means that if $F=\bigcup B_{x}$ is a canonical
representation, then for every $\theta\in S^{n-1}$ there is a unique
ball in the set $\left\{ B_{x}\right\} $ such that $x=r_{\theta}\theta$
for some $r_{\theta}\ge0$. 
\begin{definition}
\label{def:naive-f}Let $F=\bigcup_{\theta\in S^{n-1}}B_{r_{\theta}\theta}$
be any flower in its canonical representation. Consider any function
$f:[0,\infty)\to[0,\infty)$ such that $f(0)=0$. We define a new
flower $f(F)$ by 
\begin{equation}
f(F)=\bigcup_{\theta\in S^{n-1}}B_{f(r_{\theta})\theta}.\label{eq:function-to-flower}
\end{equation}
 
\end{definition}

Recall that for a given function $g:S^{n-1}\to(0,\infty)$, the Alexandrov
body of $g$ is defined by 
\[
A[g]=\left\{ x\in\RR^{n}:\ \left\langle x,\theta\right\rangle \le g(\theta)\text{ for all }\theta\in S^{n-1}\right\} .
\]
 In other words, $A[g]$ is the largest convex body with $h_{A[g]}\le g$.
We have the following simple claim:
\begin{proposition}
For every flower $F=\bigcup_{\theta\in S^{n-1}}B_{r_{\theta}\theta}$
we have $\cof\left(f(F)\right)=A\left[\frac{1}{f(r_{\theta})}\right].$ 
\end{proposition}

\begin{proof}
Recall the following property of the spherical inversion $\phi$:
For any sphere $S\subseteq\RR^{n}$ with $0\in S$ the inversion
$\phi(S)$ is an affine hyperplane. Stating the same using our notation,
for every $\theta\in S^{n-1}$ and $c\ge0$ we have $\cof\left(B_{c\theta}\right)=H\left(\theta,\frac{1}{c}\right)$,
where $H(\theta,a)$ is the half-space defined by 
\[
H\left(\theta,a\right):=\left\{ x\in\RR^{n}:\ \left\langle x,\theta\right\rangle \le a\right\} .
\]
Since $\cof$ is order reversing we conclude that indeed
\[
\cof\left(f(F)\right)=\bigcap_{\theta\in S^{n-1}}\cof\left(B_{f(r_{\theta})\theta}\right)=\bigcap_{\theta\in S^{n-1}}H\left(\theta,\frac{1}{f(r_{\theta})}\right)=A\left[\frac{1}{f(r_{\theta})}\right].
\]
\end{proof}
We may apply $f$ to convex bodies $K$ by setting $f(K)=\left(f\left(K^{\flower}\right)\right)^{-\flower}.$
We again have an equivalent description:
\begin{proposition}
\label{prop:f-convex-formula}Let $S$ be the star body with radial
function $r_{S}(\theta)=f(r_{K}(\theta))$. Then 
\[
f(K)=A\left[\frac{1}{f(r_{K}(\theta))}\right]^{\circ}=\conv S.
\]
\end{proposition}

\begin{proof}
Write $F=K^{\flower}$, then the canonical representation of $F$
is $F=\bigcup_{\theta\in S^{n-1}}B_{r_{K}(\theta)\theta}$. Hence
\[
f(K)=f(F)^{-\flower}=\left(\cof\left(A\left[\frac{1}{f(r_{K}(\theta))}\right]\right)\right)^{-\flower}=A\left[\frac{1}{f(r_{K}(\theta))}\right]^{\circ}.
\]
 Recalling that $H(\theta,c)^{\circ}=[0,\frac{1}{c}\theta]$ we conclude
that 
\begin{align*}
f(K) & =\left(\bigcap_{\theta\in S^{n-1}}H\left(\theta,\frac{1}{f(r_{K}(\theta))}\right)\right)^{\circ}=\conv\bigcup_{\theta\in S^{n-1}}H\left(\theta,\frac{1}{f(r_{K}(\theta))}\right)^{\circ}\\
 & =\conv\bigcup_{\theta\in S^{n-1}}[0,f(r_{K}(\theta))]=\conv S.
\end{align*}
\end{proof}
The above definition of $f(F)$ is a naive one. The problem is that
the representation of $f(F)$ in \eqref{eq:function-to-flower} does
not have to be canonical. As a result, for two functions $f_{1},f_{2}:[0,\infty)\to[0,\infty)$
we may have 
\begin{equation}
\left(f_{1}\circ f_{2}\right)(F)\ne f_{1}\left(f_{2}(F)\right).\label{eq:composition-broken}
\end{equation}
However, we sometimes have a one-sided inclusion:
\begin{proposition}
\label{prop:func-inclusion}If $f_{1}$ is monotone increasing then
$\left(f_{1}\circ f_{2}\right)(F)\subseteq f_{1}\left(f_{2}(F)\right)$
and $\left(f_{1}\circ f_{2}\right)(K)\subseteq f_{1}\left(f_{2}(K)\right)$
for every flower $F$ and convex body $K$.
\end{proposition}

\begin{proof}
If $S$ is the star body with $r_{S}=f_{2}\left(r_{K}\right)$ then
$f_{2}(K)=\conv S\supseteq S$. Since $f_{1}$ is increasing we have
\[
f_{1}\left(r_{f_{2}(K)}\right)\ge f_{1}\left(r_{S}\right)=f_{1}\left(f_{2}\left(r_{K}\right)\right).
\]
Therefore if $S_{1}$ has radial function $r_{S_{1}}=f_{1}\left(r_{f_{2}(K)}\right)$
and $S_{2}$ has radial function $\left(f_{1}\circ f_{2}\right)(r_{K})$
then $S_{1}\supseteq S_{2}$. Hence $f_{1}\left(f_{2}(K)\right)=\conv S_{1}\supseteq\conv S_{2}=\left(f_{1}\circ f_{2}\right)(K)$. 
\end{proof}
Despite this problem, Definition \ref{def:naive-f} is still a useful
one. For example, let us consider the function $f(x)=x^{\lambda}$
for some $0<\lambda<1$. Then the body $f(K)$ is related to the ``logarithmic
Minkowski addition'' of Böröczky, Lutwak, Yang and Zhang (see \cite{Boroczky2012}).
To be more precise, fix two convex bodies $K$ and $T$ containing
the origin , $0<\lambda<1$ and any $p>0$. Then the $p$-mean of
$K$ and $T$ is defined by 
\[
(1-\lambda)\cdot K+_{p}\lambda\cdot T=A\left[\left((1-\lambda)h_{K}^{p}+\lambda h_{T}^{p}\right)^{1/p}\right].
\]

Note that when $p>1$ the function on the right hand side is convex,
and hence is exactly equal to the support function of $(1-\lambda)\cdot K+_{p}\lambda\cdot T$,
but this is no longer the case for $0<p<1$. Taking the limit $p\to0^{+}$
we define the 0-mean, or logarithmic mean of $K$ and $T$ to be 
\[
(1-\lambda)\cdot K+_{0}\lambda\cdot T=A\left[h_{K}^{1-\lambda}h_{T}^{\lambda}\right].
\]
Then, using Proposition \ref{prop:f-convex-formula} it easy to check
that if $f(x)=x^{\lambda}$ then 
\[
f(K)=\left((1-\lambda)\cdot B_{2}^{n}+_{0}\lambda\cdot K^{\circ}\right)^{\circ}.
\]
 Such ``dual logarithmic means'' were studied by Saroglou \cite{Saroglou2016}.
Using our notation, he showed the following:
\begin{theorem}[Saroglou, \cite{Saroglou2016} ]
\label{thm:dual-BMI}For every convex body $K$ containing the origin
we have $\left|f(K)\right|\le\left|B_{2}^{n}\right|^{1-\lambda}\left|K\right|^{\lambda}$. 
\end{theorem}

The dual question, asking for a lower bound on the volume of $(1-\lambda)\cdot K+_{0}\lambda\cdot T$,
is an open problem in convexity known as the log-Brunn-Minkowski conjecture
which was introduced in \cite{Boroczky2012}. We will not discuss
it further here.

\section{Power Functions}

Before we continue to develop a general approach for constructing
``functions of convex bodies'', let us ``correct'' the problem
in \eqref{eq:composition-broken} for the family of functions $f_{\lambda}(x)=x^{\lambda}$,
i.e. let us build a (convex body valued) power function $K^{\lambda}$
that has the semigroup property $\left(K^{\lambda}\right)^{\mu}=K^{\lambda\mu}$.
This will be done not for every positive value of $\lambda$, as we
will see:
\begin{theorem}
On the class of flowers there are maps $F\mapsto F^{\lambda}$, $0\le\lambda<\infty$,
with the following properties:
\begin{enumerate}
\item $F^{1}=F$ and $F^{0}=B_{2}^{n}$. 
\item If $F_{1}\subseteq F_{2}$ then $F_{1}^{\lambda}\subseteq F_{2}^{\lambda}$.
\item $\left(tF\right)^{\lambda}=t^{\lambda}F^{\lambda}$ for $t\ge0$.
\item $F^{\lambda}$ is continuous with respect to both $F$ and $\lambda$. 
\item $\left(F^{\mu}\right)^{\lambda}=F^{\lambda\mu}$ for $0<\lambda,\mu\le1$
and for $1\le\lambda,\mu\le\infty$. 
\item If $aB_{2}^{n}\subseteq F\subseteq\sqrt{2}aB_{2}^{n}$ for some $a>0$
then $\left(F^{\lambda}\right)^{1/\lambda}=F$ for all $0<\lambda\le1$. 
\end{enumerate}
\end{theorem}

The construction of $F^{\lambda}$ and the proofs of properties (1)-(5)
are similar to the construction described in \cite{Milman2017a} and
\cite{Milman2017}. We sketch it here:
\begin{proof}
Let us write $P_{\lambda}(F)$ for the body $f(F)$, where $f:[0,\infty)\to[0,\infty)$
is the function $f(x)=x^{\lambda}$. From Proposition \ref{prop:func-inclusion}
we know that 
\begin{equation}
P_{\lambda}\left(P_{\mu}(F)\right)\supseteq P_{\lambda\mu}(F),\label{eq:prepower-ineq}
\end{equation}
but usually there will not be an equality. 

To define the proper power $F^{\lambda}$ we first define $F^{1}=F$
an $F^{0}=B_{2}^{n}$. Assume now that $0<\lambda<1$. Fix any partition
$\Pi=\left\{ t_{0},t_{1},t_{2},\ldots,t_{m}\right\} $ of the interval
$[\lambda,1]$, by which we mean we fix numbers $t_{0},t_{1},\ldots,t_{m}$
such that
\[
\lambda=t_{0}<t_{1}<\cdots<t_{m}=1
\]

We then define $s_{i}=t_{i-1}/t_{i}$ for $i=1,2,..m$ and write
\[
P_{\Pi}(F)=\left(P_{s_{1}}\circ P_{s_{2}}\circ\cdots\circ P_{s_{m}}\right)\left(F\right).
\]
The inclusion \eqref{eq:prepower-ineq} implies that if $\widetilde{\Pi}\supseteq\Pi$
then $P_{\widetilde{\Pi}}(F)\supseteq P_{\Pi}(F)$. We then define
\[
F^{\lambda}=\overline{\bigcup_{\Pi}P_{\lambda}(F)},
\]
where the union is taken over all partitions of $[0,1]$. 

Let us write $\left\Vert \Pi\right\Vert =\max_{i}\left|t_{i+1}-t_{i}\right|$
for the length of the largest interval in $\Pi$. The very useful
observation is that we actually have $\lim_{\left\Vert \Pi\right\Vert \to0}P_{\Pi}(F)=F^{\lambda}$
in the following sense: For every $\epsilon>0$ there exists $\delta>0$
such that for every partition $\Pi$ of $[\lambda,1]$ with $\left\Vert \Pi\right\Vert <\delta$
one has 
\[
(1-\epsilon)F^{\lambda}\subseteq P_{\Pi}(F)\subseteq F^{\lambda}.
\]
 The proof is almost identical to the one that appeared in \cite{Milman2017},
so we will not reproduce it here. However, we will quickly recall
why this construction has properties (1)-(5). 

The proof of properties (1), (2) and (3) are trivial. Indeed, (1)
is just a definition. For (2) and (3), $P_{\lambda}(F)$ satisfies
these properties by definition, hence $P_{\Pi}(F)$ satisfies them,
and by taking the limit $\left\Vert \Pi\right\Vert \to0$ we see they
are satisfied by $F^{\lambda}$. 

To show property (5), fix $0<\lambda,\mu\le1$. Fix a partition $\Pi_{\lambda}$
of $\left[\lambda,1\right]$ and a partition $\Pi_{\mu}$ of $\left[\mu,1\right]$.
If $\Pi_{\lambda}=\left\{ t_{0},t_{1},\ldots,t_{m}\right\} $ we define
$\mu\Pi_{\lambda}=\left\{ \mu t_{0},\mu t_{1},\ldots,\mu t_{m}\right\} $.
Note that this is a partition of $\left[\lambda\mu,\mu\right]$, so
$\Pi=\mu\Pi_{\lambda}\cup\Pi_{\mu}$ is a partition of $\left[\lambda\mu,1\right]$.
Immediately from the definition we have 
\[
P_{\Pi}(F)=P_{\mu\Pi_{\lambda}}\left(P_{\Pi_{\mu}}\left(F\right)\right)=P_{\Pi_{\lambda}}\left(P_{\Pi_{\mu}}(F)\right).
\]
 When $\left\Vert \Pi_{\lambda}\right\Vert \to0$ and $\left\Vert \Pi_{\mu}\right\Vert \to0$
we also have $\left\Vert \Pi\right\Vert \to0$, so we get $F^{\lambda\mu}=\left(F^{\mu}\right)^{\lambda}$
like we wanted.

Finally we prove property (4). Let $F=\bigcup_{\theta\in S^{n-1}}B_{r_{\theta}\theta}$
be a flower in its canonical representation. Fix a number $R>0$ such
that $\frac{1}{R}\le r_{\theta}\le R$ for all $\theta\in S^{n-1}$.
Then for every $0<\lambda,\mu<1$ we have
\[
P_{\lambda}(F)=\bigcup_{\theta\in S^{n-1}}B_{r_{\theta}^{\lambda}\theta}\subseteq\bigcup_{\theta\in S^{n-1}}B_{R^{\left|\mu-\lambda\right|}r_{\theta}^{\mu}\theta}=R^{\left|\mu-\lambda\right|}P_{\mu}(F).
\]
 It follows in the usual way that $F^{\lambda}\subseteq R^{\left|\mu-\lambda\right|}F^{\mu}$.
Note that the same remains if $\lambda$ are $\mu$ allowed to take
the values $0$ and $1$ (recall that we defined $F^{1}=F$ and $F^{0}=B_{2}^{n}$). 

Fix sequences $\left\{ F_{i}\right\} $, $\left\{ \lambda_{i}\right\} $
such that $F_{i}\to F$ and $\lambda_{i}\to\lambda$. Since $F_{i}\to F$
there exists a sequence $\left\{ \epsilon_{i}\right\} $ such that
$\epsilon_{i}\to0$ and $(1-\epsilon_{i})F\subseteq F_{i}\subseteq(1+\epsilon_{i})F$.
Hence 

\[
F_{i}^{\lambda_{i}}\subseteq(1+\epsilon_{i})^{\lambda_{i}}F^{\lambda_{i}}\subseteq(1+\epsilon_{i})^{\lambda_{i}}R^{\left|\lambda-\lambda_{i}\right|}F^{\lambda}\xrightarrow{i\to\infty}F^{\lambda},
\]
 and 
\[
F_{i}^{\lambda_{i}}\supseteq(1-\epsilon_{i})^{\lambda_{i}}F^{\lambda_{i}}\supseteq(1-\epsilon_{i})^{\lambda_{i}}R^{-\left|\lambda-\lambda_{i}\right|}F^{\lambda}\xrightarrow{i\to\infty}F^{\lambda},
\]
 so $F_{i}^{\lambda_{i}}\to F^{\lambda}$ and the power maps are continuous. 

So far we have only discussed the case $\lambda<1$, but the construction
in the case $\lambda>1$ is almost identical. This time we fix a partition
$\Pi=\left\{ t_{0},t_{1},t_{2},\ldots,t_{m}\right\} $ of $[1,\lambda]$,
set $s_{i}=t_{i}/t_{i-1}$ for $i=1,..m$ and define 
\[
P_{\Pi}(F)=\left(P_{s_{m}}\circ P_{s_{m-1}}\circ\cdots\circ P_{s_{1}}\right)\left(F\right).
\]
 Again one defines $F^{\lambda}=\overline{\bigcup_{\Pi}P_{\lambda}(F)}=\lim_{\left\Vert \Pi\right\Vert \to0}P_{\Pi}(F)$.
The proof of all properties is completely analogous. 
\end{proof}
To prove (6), note that this theorem is equivalent to a similar theorem
about power functions on the family of convex bodies. Since we assumed
our flowers are compact with the origin at their interior, we assume
the same for our convex bodies:
\begin{theorem}
\label{thm:power-bodies}On the class of compact convex bodies with
$0$ in their interior there are maps $K\mapsto K^{\lambda}$, $0\le\lambda<\infty$,
with the following properties:
\begin{enumerate}
\item $K^{1}=K$ and $K^{0}=B_{2}^{n}$. 
\item If $K_{1}\subseteq K_{2}$ then $K_{1}^{\lambda}\subseteq K_{2}^{\lambda}$.
\item $\left(tK\right)^{\lambda}=t^{\lambda}K^{\lambda}$ for $t\ge0$.
\item $K^{\lambda}$ is continuous with respect to both $K$ and $\lambda$.
\item $\left(K^{\lambda}\right)^{\mu}=K^{\lambda\mu}$ for $0\le\lambda,\mu\le1$
and for $1\le\lambda,\mu\le\infty$. 
\item If $aB_{2}^{n}\subseteq K\subseteq\sqrt{2}aB_{2}^{n}$ for some $a>0$
then $\left(K^{\lambda}\right)^{1/\lambda}=K$ for all $0<\lambda\le1$. 
\end{enumerate}
\end{theorem}

We will prove property (6) in the language of convex bodies. The following
proposition will be needed:
\begin{proposition}
For $K\in\KK_{0}$ and $\lambda>0$ let $S_{\lambda}(K)$ be the star
body with radial function $r_{S_{\lambda}(K)}=r_{K}^{\lambda}$. If
$aB_{2}^{n}\subseteq K\subseteq\sqrt{2}aB_{2}^{n}$ for some $a>0$
then $S_{\lambda}(K)$ is convex for all $0<\lambda\le1$. 
\end{proposition}

\begin{proof}
Note that for every subspace $E\subseteq\RR^{n}$ we have $S_{\lambda}\left(K\cap E\right)=S_{\lambda}K\cap E$.
It is therefore enough to prove the Proposition in dimension $n=2$.
By dilating $K$ we may assume that $a=1$. By standard approximation
we may also assume that $K$ is a $C^{2}$ convex body with positive
curvature at every boundary point (see, e.g. Section 27 of \cite{Bonnesen1987}). 

Let $G_{K}:\RR^{2}\to\RR$ be the gauge function of $K$, i.e. $G$
is $1$-homogeneous and $G_{K}(\theta)=\frac{1}{r_{K}(\theta)}$ for
all $\theta\in S^{1}$. Define $g_{K}:\RR\to\RR$ by $g_{K}(t)=G_{K}\left(\cos t,\sin t\right)$.
Then the convexity of $K$ is equivalent to $G$ being a convex function,
i.e. $\nabla^{2}G_{K}\succeq0$. This is equivalent to the requirement
that $g_{K}^{\prime\prime}+g_{K}\ge0$.

Obviously $g_{S_{\lambda}(K)}=g_{K}^{\lambda}$. Differentiating we
have the formula 
\begin{align*}
\frac{\left(g_{K}^{\lambda}\right)^{\prime\prime}}{g_{K}^{\lambda}}+1 & =\lambda\left(\frac{g_{K}^{\prime\prime}}{g_{K}}+1\right)+(1-\lambda)\left(1-\lambda\left(\frac{g_{K}^{\prime}}{g_{K}}\right)^{2}\right)\\
 & \ge(1-\lambda)\left(1-\lambda\left(\frac{g_{K}^{\prime}}{g_{K}}\right)^{2}\right).
\end{align*}

To bound this expression, fix a point $\theta=\left(\cos t,\sin t\right)$
and let $n(\theta)$ be the unit normal to $K$ at the point $p=r_{K}(\theta)\theta$.
Then the half space
\[
H=\left\{ x\in\RR^{2}:\ \left\langle x,n(\theta)\right\rangle \le\left\langle p,n(\theta)\right\rangle \right\} 
\]
must satisfy $H\supseteq K\supseteq B_{2}^{n}$. Hence $n(\theta)\in H$,
so 
\[
\left\langle \theta,n(\theta)\right\rangle =\frac{1}{r_{K}(\theta)}\left\langle p,n(\theta)\right\rangle \ge\frac{1}{\sqrt{2}}\cdot\left\langle n(\theta),n(\theta)\right\rangle =\frac{1}{\sqrt{2}}.
\]
 Since $n(\theta)=\frac{\nabla G_{K}(\theta)}{\left|\nabla G_{K}(\theta)\right|}$
is follows that 
\[
\left|\nabla G_{K}(\theta)\right|^{2}=\frac{\left\langle \nabla G_{K}(\theta),\theta\right\rangle ^{2}}{\left\langle n(\theta),\theta\right\rangle ^{2}}\le\frac{G_{K}(\theta)^{2}}{1/2}=2G_{K}(\theta)^{2},
\]
 where we used the fact that $G_{K}$ is 1-homogeneous so $\left\langle \nabla G_{K}(\theta),\theta\right\rangle =G_{K}(\theta)$.
If we now write $\eta=(-\sin t,\cos t)$ then $\left\{ \theta,\eta\right\} $
is an orthonormal basis of $\RR^{2}$, so 
\[
g_{K}^{\prime}(t)^{2}=\left\langle \nabla G_{K}(\theta),\eta\right\rangle ^{2}=\left|\nabla G_{K}(\theta)\right|^{2}-\left\langle \nabla G_{K}(\theta),\theta\right\rangle ^{2}\le2G_{K}(\theta)^{2}-G_{K}(\theta)^{2}=g_{K}(t)^{2}.
\]
 Hence for every $0<\lambda\le1$ we have 
\[
\frac{\left(g_{K}^{\lambda}\right)^{\prime\prime}}{g_{K}^{\lambda}}+1\ge(1-\lambda)\left(1-\lambda\left(\frac{g_{K}^{\prime}}{g_{K}}\right)^{2}\right)\ge(1-\lambda)\left(1-\lambda\right)\ge0,
\]
 so $S_{\lambda}(K)$ is indeed convex. 
\end{proof}
And now we can prove:
\begin{proof}[Proof of Theorem \ref{thm:power-bodies} (part 6).]
First we claim that for every $0<\lambda\le1$ we have $K^{\lambda}=S_{\lambda}(K)$.
To see this, let us translate the definition of $K^{\lambda}$ from
the language of flowers to the language of convex bodies. Define $P_{\lambda}(K)=\conv S_{\lambda}(K)$.
By Proposition \ref{prop:f-convex-formula} we have $P_{\lambda}(K)=P_{\lambda}(K^{\flower})^{-\flower}$.
For every partition 
\[
\Pi:\ \lambda=t_{0}<t_{1}<\cdots<t_{m}=1
\]
 of the interval $[\lambda,1]$ we set $s_{i}=t_{i-1}/t_{i}$ for
$i=1,2,..m$ and define
\[
P_{\Pi}(K)=\left(P_{s_{1}}\circ P_{s_{2}}\circ\cdots\circ P_{s_{m}}\right)\left(K\right),
\]
 and then $K^{\lambda}=\lim_{\left\Vert \Pi\right\Vert \to0}P_{\Pi}(K)$. 

However, in our case $S_{\mu}(K)$ is convex for all $0<\mu\le1$,
so $P_{\mu}(K)=S_{\mu}(K)$. Since $S_{\mu_{1}\mu_{2}}=S_{\mu_{1}}\circ S_{\mu_{2}}$
for every $\mu_{1},\mu_{2}>0$ we see immediately that $P_{\Pi}(K)=S_{\lambda}(K)$
for every partition $\Pi$ of $\left[\lambda,1\right]$. Hence $K^{\lambda}=S_{\lambda}(K)$
as well.

Now the definition of $K^{\mu}$ for $\mu>1$ is essentially the same:
we take a partition 
\[
\Pi:\ 1=t_{0}<t_{1}<\cdots<t_{m}=\mu
\]
 of $[1,\mu]$, set $s_{i}=t_{i}/t_{i-1}$ and define 
\[
P_{\Pi}(K)=\left(P_{s_{m}}\circ\cdots\circ P_{s_{2}}\circ P_{s_{1}}\right)(K).
\]
 In particular we see that for every $\mu\le\frac{1}{\lambda}$ and
every partition $\Pi$ of $\left[1,\mu\right]$ we have 
\begin{align*}
P_{\Pi}\left(K^{\lambda}\right) &= \left(P_{s_{m}}\circ\cdots\circ P_{s_{2}}\circ P_{s_{1}}\right)(S_{\lambda}(K)) \\
&=\left(S_{s_{m}}\circ\cdots\circ S_{s_{2}}\circ S_{s_{1}}\right)(S_{\lambda}(K))=S_{\lambda\mu}(K)=K^{\lambda\mu}.
\end{align*}
 It follows that $\left(K^{\lambda}\right)^{\mu}=K^{\lambda\mu}$.
In particular for $\mu=\frac{1}{\lambda}$ we have $\left(K^{\lambda}\right)^{1/\lambda}=K$. 
\end{proof}
We should note that the two cases $\lambda<1$ and $\lambda>1$ are
dramatically different. For example, let $K=\left[-\frac{1}{\sqrt{n}},\frac{1}{\sqrt{n}}\right]^{n}$
be the cube with $2^{n}$ vertices on the unit sphere. Then for all
$\lambda>1$ we have $K^{\lambda}=K$, while for any $0<\lambda_{1}<\lambda_{2}<1$
we have $K^{\lambda_{1}}\supsetneq K^{\lambda_{2}}$. To see this
we prove the following:
\begin{proposition}
\label{prop:power-compare}For every $0<\lambda\le1$ and every $K\in\KK_{0}$
we have $r_{K^{\lambda}}\ge r_{K}^{\lambda}$ and $h_{K^{\lambda}}\le h_{K}^{\lambda}$.
\end{proposition}

\begin{proof}
The inequality $r_{K^{\lambda}}\ge r_{K}^{\lambda}$ is trivial and
in fact holds for every $\lambda>0$: by definition we have 
\[
r_{P_{\lambda}(K)}\ge r_{S_{\lambda}(K)}=r_{K}^{\lambda},
\]
 hence $r_{P_{\Pi}(K)}\ge r_{K}^{\lambda}$ for every partition $\Pi$
of $[\lambda,1]$ (or of $[1,\lambda]$ in the case $\lambda>1$),
so the result follows.

For the other inequality it is again enough to show that $h_{S_{\lambda}(K)}\le h_{K}^{\lambda}$.
Note that $S_{\lambda}(K)$ is not necessarily convex, but we can
define its support function in the usual way as $h_{S_{\lambda}(K)}(\theta)=\sup_{x\in S_{\lambda}(K)}\left\langle x,\theta\right\rangle $.
Indeed, fix a point $x\in S_{\lambda}(K)$ and write $x=\rho\cdot\eta$
for some $\rho>0$ and $\eta\in S^{n-1}$. Then 
\[
\rho\le r_{S_{\lambda}(K)}(\eta)=r_{K}(\eta)^{\lambda}.
\]
 If $\left\langle \eta,\theta\right\rangle <0$ then obviously $\left\langle x,\theta\right\rangle \le0\le h_{K}(\theta)^{\lambda}$.
If on the other hand $\left\langle \eta,\theta\right\rangle \ge0$
then 
\[
\left\langle x,\theta\right\rangle =\rho\cdot\left\langle \eta,\theta\right\rangle \le r_{K}(\eta)^{\lambda}\left\langle \eta,\theta\right\rangle \le\left\langle r_{K}(\eta)\eta,\theta\right\rangle ^{\lambda}\le h_{K}(\theta)^{\lambda},
\]
 where the last inequality holds since $r_{K}(\eta)\eta\in K$. We
conclude that indeed $h_{S_{\lambda}(K)}\le h_{K}^{\lambda}$.

Now we finish the proof like in the case of the radial functions:
we have $h_{P_{\lambda}(K)}=h_{S_{\lambda}(K)}\le h_{K}^{\lambda}$,
hence $h_{P_{\Pi}(K)}\le h_{K}^{\lambda}$ for every partition $\Pi$
of $[\lambda,1]$ , hence $h_{K^{\lambda}}\le h_{K}^{\lambda}$.
\end{proof}
We therefore have:
\begin{proposition}
If $K^{\lambda_{1}}=K^{\lambda_{2}}$ for $\lambda_{1},\lambda_{2}\le1$,
$\lambda_{1}\ne\lambda_{2}$, then $K=B_{2}^{n}$. 
\end{proposition}

\begin{proof}
Fix directions $\theta_{1},\theta_{2}\in S^{n-1}$ such that $\theta_{1}=\max_{\theta\in S^{n-1}}r_{K}(\theta)$
and $\theta_{2}=\min_{\theta\in S^{n-1}}r_{K}(\theta)$. The supporting
hyperplane in direction $\theta_{i}$ must be orthogonal to $\theta_{i}$,
so $h_{K}(\theta_{i})=r_{K}(\theta_{i})$. But then for every $0<\lambda\le1$
we have
\[
h_{K^{\lambda}}(\theta_{i})\le h_{K}(\theta_{i})^{\lambda}=r_{K}(\theta_{i})^{\lambda}\le r_{K^{\lambda}}(\theta_{i})\le h_{K^{\lambda}}(\theta_{i}),
\]
 so we must have $r_{K^{\lambda}}(\theta_{i})=r_{K}(\theta_{i})^{\lambda}$.

Since we assumed that $K^{\lambda_{1}}=K^{\lambda_{2}}$ we have $r_{K}(\theta_{i})^{\lambda_{1}}=r_{K}(\theta_{i})^{\lambda_{2}}$,
which can only happen if $r_{K}(\theta_{1})=r_{K}(\theta_{2})=1$.
By definition of $\theta_{1}$ and $\theta_{2}$ we have $r_{K}(\theta)=1$
for all $\theta\in S^{n-1}$, so $K=B_{2}^{n}$. 
\end{proof}
Let us now discuss the volume of the bodies $K^{\lambda}$. For $0<\lambda<1$,
they satisfy the same inequality as Saroglou's Theorem \ref{thm:dual-BMI}:
\begin{theorem}
For every convex body $K$ and $0<\lambda<1$ we have $\left|K^{\lambda}\right|\le\left|B_{2}^{n}\right|^{1-\lambda}\left|K\right|^{\lambda}$.
\end{theorem}

Note that Sargolou's result can be stated as $\left|P_{\lambda}(K)\right|\le\left|B_{2}^{n}\right|^{1-\lambda}\left|K\right|^{\lambda}$,
where $P_{\lambda}(K)$ was the naive application of $f(x)=x^{\lambda}$
to the body $K$. Since $P_{\lambda}(K)\subseteq K^{\lambda}$, this
Theorem is formally stronger than Theorem \ref{thm:dual-BMI}. However,
they are actually equivalent:
\begin{proof}
Fix $0<\lambda,\mu<1$. Using Theorem \ref{thm:dual-BMI} twice we
see that
\[
\left|P_{\lambda}\left(P_{\mu}(K)\right)\right|\le\left|B_{2}^{n}\right|^{1-\lambda}\left|P_{\mu}(K)\right|^{\lambda}\le\left|B_{2}^{n}\right|^{1-\lambda}\left(\left|B_{2}^{n}\right|^{1-\mu}\left|K\right|^{\mu}\right)^{\lambda}=\left|B_{2}^{n}\right|^{1-\lambda\mu}\left|K\right|^{\lambda\mu}.
\]
 Iterating, we see that if $\Pi$ is any partition of $\left[\lambda,1\right]$
then $\left|P_{\Pi}(K)\right|\le\left|B_{2}^{n}\right|^{1-\lambda}\left|K\right|^{\lambda}$.
Taking the limit as $\left\Vert \Pi\right\Vert \to0$ the result follows.
\end{proof}
For $\lambda>1$ the reverse inequality is true, as is easy to prove
directly:
\begin{proposition}
For every convex body $K$ and $\lambda>1$ we have $\left|K^{\lambda}\right|\ge\left|B_{2}^{n}\right|^{1-\lambda}\left|K\right|^{\lambda}$.
\end{proposition}

\begin{proof}
Using integration in polar coordinates we know that for every star
body $A$ we have $\left|A\right|=\left|B_{2}^{n}\right|\int_{S^{n-1}}r_{A}^{n}\dd\sigma$,
where $\sigma$ denotes the uniform probability measure on $S^{n-1}$.
In Proposition \ref{prop:power-compare} we saw that $r_{K^{\lambda}}\ge r_{K}^{\lambda}$
even in the case $\lambda>1$. Using this fact and Jensen's inequality
we immediately obtain
\begin{align*}
\left|K^{\lambda}\right| & =\left|B_{2}^{n}\right|\int_{S^{n-1}}r_{K^{\lambda}}^{n}\dd\sigma\ge\left|B_{2}^{n}\right|\int_{S^{n-1}}\left(r_{K}^{n}\right)^{\lambda}\dd\sigma\ge\left|B_{2}^{n}\right|\left(\int_{S^{n-1}}r_{K}^{n}\dd\sigma\right)^{\lambda}\\
 & =\left|B_{2}^{n}\right|^{1-\lambda}\left(\left|B_{2}^{n}\right|\int_{S^{n-1}}r_{K}^{n}\dd\sigma\right)^{\lambda}=\left|B_{2}^{n}\right|^{1-\lambda}\left|K\right|^{\lambda}.
\end{align*}
\end{proof}
Finally, let us mention that in \cite{Milman2017a} and \cite{Milman2018}
we had another definition of the power $K^{\lambda}$ for $0<\lambda<1$,
which had the additional advantage of interacting well with the polarity
map, in the sense that $\left(K^{\circ}\right)^{\lambda}=\left(K^{\lambda}\right)^{\circ}$.
More generally, we defined the ``weighted geometric mean'' of two
convex bodies in a way that commutes with the polarity map, i.e. $g_{\lambda}(K,T)^{\circ}=g_{\lambda}(K^{\circ},T^{\circ})$.
However this definition was not explicit, and relied on the existence
of ultra-filters. We have no reason to expect the two definitions
to be the same, even though explicit examples are very difficult to
compute. In particular, we do not have similar inequalities regarding
the volume of $K^{\lambda}$ for this other definition. 

\section{Composition of convex bodies}

Return to the general construction of $f(K)$. Consider a function
$f:S^{n-1}\times[0,\infty)\to[0,\infty)$ of two variables. Let $F=\bigcup_{\theta}B_{r_{\theta}\theta}$
be a flower in its canonical representation. We define
\[
f(F)=\bigcup_{\theta\in S^{n-1}}B_{f(\theta,r_{\theta})\theta}.
\]
 As before, we have $\cof\left(f(F)\right)=A\left[\frac{1}{f(\theta,r_{\theta})}\right]$.
For the core $K=F^{-\flower}$ we can again define 
\[
f(K)=f(F)^{-\flower}=f\left(K^{\flower}\right)^{-\flower}.
\]
 Again we have 
\[
f(K)=A\left[\frac{1}{f(\theta,r_{\theta})}\right]^{\circ}=\conv S,
\]
 where $S$ is a star body with radial function $r_{S}(\theta)=f(\theta,r_{\theta})$. 

We will use these formulas to define composition of two convex bodies
$K$ and $T$. Again, we are doing it working with flowers. Consider
two flowers in canonical representations
\[
F_{1}=\bigcup_{\theta\in S^{n-1}}B_{\rho_{1}(\theta)\theta}\text{\ensuremath{\quad\text{and}}}\quad F_{2}=\bigcup_{\theta\in S^{n-1}}B_{\rho_{2}(\theta)\theta}.
\]
 We use $\rho_{1},\rho_{2}$ instead of $r_{1}$ and $r_{2}$ to avoid
confusion between the function $\rho_{i}$ and the radial function
$r_{F_{i}}$. Define 
\[
F_{1}\circ F_{2}=r_{F_{1}}(F_{2}),
\]
 where we think of the radial function $r_{F_{1}}$ as a function
of two variables $r_{F_{1}}:S^{n-1}\times[0,\infty)\to[0,\infty)$
by setting $r_{F_{1}}(\theta,t)=r_{F_{1}}(\theta)\cdot t$. More explicitly,
it means that 
\[
F_{1}\circ F_{2}=\bigcup_{\theta\in S^{n-1}}B_{r_{F_{1}}(\theta,\rho_{2}(\theta))\theta}=\bigcup_{\theta\in S^{n-1}}B_{r_{F_{1}}(\theta)\rho_{2}(\theta)\theta}.
\]

Consider now convex bodies $K$ and $T$. Let $F_{1}=T^{\flower}$
and $F_{2}=K^{\flower}$. Then $\rho_{2}(\theta)=r_{K}(\theta)$,
the radial function of $K$, while $r_{F_{1}}(\theta)=h_{T}(\theta)$.
Therefore, if we extend the definition of $\circ$ to convex bodies
in the usual way, we get 
\begin{equation}
T\circ K=\left(\bigcup_{\theta\in S^{n-1}}B_{h_{T}(\theta)r_{K}(\theta)\theta}\right)^{-\flower}.\label{eq:body-composition}
\end{equation}
 Note an interesting example: For every convex body $T$ we have $T\circ T^{\circ}=B_{2}^{n}$,
because $h_{T}\cdot r_{T^{\circ}}=1$. And we always have 
\[
\left(T\circ K\right)^{\circ}=A\left[\frac{1}{h_{T}\cdot r_{K}}\right].
\]
 The formula \eqref{eq:body-composition} may also be seen as 
\[
T\circ K=h_{T}(K)=\left[h_{T}\left(K^{\flower}\right)\right]^{-\flower}.
\]
 In this form we see that we may also use another function instead
of $h_{T}$ which is naturally connected with $T$\textendash{} its
radial function $r_{T}$. Then a different composition (the radial
composition) will be 
\[
T\odot K:=r_{T}(K)=r_{T}\left(K^{\flower}\right)^{-\flower}=\left(\bigcup_{\theta\in S^{n-1}}B_{r_{T}(\theta)r_{K}(\theta)\theta}\right)^{-\flower}.
\]
 This is a commutative operation, a kind of ``product'' on the class
of convex bodies. 

Note that if $T=B_{2}^{n}$ then both compositions preserve $K$,
i.e. this is the identity map on $\KK_{0}$.
\begin{problem}
For which bodies $T$ do we have a volume inequality of the form 
\[
\left|T\circ\left(K_{1}+K_{2}\right)\right|^{\frac{1}{n}}\ge\left|T\circ K_{1}\right|^{\frac{1}{n}}+\left|T\circ K_{2}\right|^{\frac{1}{n}}
\]
 or 
\[
\left|T\odot\left(K_{1}+K_{2}\right)\right|^{\frac{1}{n}}\ge\left|T\odot K_{1}\right|^{\frac{1}{n}}+\left|T\circ K_{2}\right|^{\frac{1}{n}}?
\]
\end{problem}

For example, when $T=B_{2}^{n}$ these inequalities are obviously
true, as they reduce to the standard Brunn-Minkowski inequality. Hence
one may study e.g. the case when $T$ is very close to the Euclidean
ball. 

The radial composition may be rewritten in an explicit form. Define
``radial product'' $K\cdot T$ to be the star body with radial function
$r_{T\cdot K}(\theta)=r_{T}(\theta)r_{K}(\theta)$. Then just like
in Proposition \ref{prop:f-convex-formula} we have 
\[
T\odot K=\conv\left(T\cdot K\right).
\]

Using the same notation $A\cdot B$ for general star bodies we see
that 
\[
T\circ K=\conv\left(T^{\flower}\cdot K\right),
\]
because the radial function of $T^{\flower}$ is $h_{T}(\theta)$. 

\section{Bodies not containing the origin}

In the introduction we saw the relation $\FF=\cof\left(\KK_{0}\right)$.
In other words, applying the co-spherical inversion $\cof$ to convex
bodies \emph{containing the origin} we exactly obtain the class of
flowers. In this section we will discuss applying the spherical inversion
for convex bodies that do \emph{not} contain the origin. More specifically,
one of the questions we want to study is the following: Let $K$ be
a closed convex set $0\notin K$. What can be said about the convexity
of $\phi(K)$, where $\phi$ denotes the spherical inversion? Note
that since $0\notin K$ the set $\phi(K)$ is bounded, so it makes
sense to study $\phi(K)$ and not $\cof(K)$. For example, it is well
known that if $B$ is any ball not containing the origin, then $\phi(B)$
is also a ball, and in particular convex. 

We begin with a simple criterion for the convexity of $\phi(A)$ for
any set $A$.
\begin{definition}
Fix $x,y\in\RR^{n}$ such that $x,y\ne0$. Let $S$ be the unique
(one dimensional) circle passing through $x$, $y$ and $0$. Then
$\left(x,y\right)$ will denote the arc between $x$ and $y$ along
$S$ that does not pass through $0$. If $x=y$ we set $(x,x)=\left\{ x\right\} $.
\end{definition}

\begin{proposition}
\label{prop:arc-inversion}For any set $A\subseteq\RR^{n}$, the set
$\phi(A)$ is convex if and only if for every $x,y\in A$ we have
$(x,y)\subseteq A$. 
\end{proposition}

\begin{proof}
$\phi(A)$ is convex if and only if for every two points $z,w\in\phi(A)$
we have $[z,w]\subseteq\phi(A)$. This is equivalent to 
\[
\left(\phi(z),\phi(w)\right)=\phi\left(\left[z,w\right]\right)\subseteq\phi\left(\phi(A)\right)=A.
\]
 Since $x=\phi(z)$ and $y=\phi(w)$ are arbitrary points in $A$
the proof is complete. 
\end{proof}
In order to continue our discussion we will need the following definitions:
\begin{definition}
Let $A\subseteq\RR^{n}$ be a closed set such that $0\notin A$. Then:
\begin{enumerate}
\item The outer cone of $A$, or out-cone of $A$ for short, is 
\[
\outc A=\bigcup_{\lambda\ge1}\lambda A.
\]
\item The inner cone of $A$, or in-cone of $A$ for short, is 
\[
\inc A=\bigcup_{0<\lambda\le1}\lambda A.
\]
\item We say that $A$ is an outer cone if $A=\outc A$, and similarly for
inner cones. 
\end{enumerate}
\end{definition}

Let $K$ be a closed convex set with $0\notin K$. We then have 
\[
K=\outc K\cap\inc K,
\]
 and therefore
\[
\phi(K)=\phi\left(\outc K\right)\cap\phi(\inc K).
\]
 Furthermore, it is easy to see that 
\[
\phi\left(\outc K\right)=\inc\left(\phi(K)\right)
\]
 and 
\[
\phi\left(\inc K\right)=\outc\left(\phi(K)\right).
\]
So $\phi$ exchanges out-cones and in-cones. Therefore the question
about the convexity of $\phi(K)$ splits naturally into two questions:
\begin{lemma}
\label{lem:in-out}$\phi(K)$ is convex if and only if $\inc\left(\phi(K)\right)$
and $\outc\left(\phi(K)\right)$ are both convex.
\end{lemma}

\begin{proof}
One direction is obvious: If $\inc\left(\phi(K)\right)$ and $\outc\left(\phi(K)\right)$
are convex so is their intersection, which is exactly $\phi(K)$. 

Conversely, assume $\phi(K)$ is convex, and fix $x,y\in\inc\left(\phi(K)\right)$.
By definition there are $a,b\ge1$ such that $ax,by\in\phi(K)$. For
every $0<\lambda<1$ we have
\[
(1-\lambda)x+\lambda y=\underbrace{\left(\frac{1-\lambda}{a}+\frac{\lambda}{b}\right)}_{\mu}\cdot\underbrace{\left(\frac{\frac{1-\lambda}{a}}{\frac{1-\lambda}{a}+\frac{\lambda}{b}}ax+\frac{\frac{\lambda}{b}}{\frac{1-\lambda}{a}+\frac{\lambda}{b}}by\right)}_{z}.
\]
 Since $\phi(K)$ is convex we know that $z\in\phi(K)$, and since
$\mu\le1$ it follows that $\mu z\in\inc\left(\phi(K)\right)$. Hence
$\inc\left(\phi(K)\right)$ is indeed convex. The proof for $\outc\left(\phi(K)\right)$
is the same. 
\end{proof}
It turns out that there is no symmetry between out-cones and in-cones.
Of the two conditions, one is satisfied automatically:
\begin{theorem}
If $C$ be a convex out-cone then $\phi(C)$ is a convex in-cone.
However, the converse is false: there exists a convex in-cone $C_{0}$
such that $\phi(C_{0})$ is not convex. 
\end{theorem}

\begin{proof}
Let $C$ be a convex out-cone and fix $x,y\in C$. Note that every
point $z\in(x,y)$ can be written as $z=\lambda w$ for $w\in[x,y]$
and $\lambda\ge1$. Since $C$ is convex it follows that $w\in C$,
and since $C$ is an out-cone we have $z=\lambda w\in C$. Hence $(x,y)\subseteq C$,
so by Proposition \ref{prop:arc-inversion} it follows that $\phi(C)$
is indeed convex. 

To show that the converse is false there are many possible counter-examples.
For example, fix any affine hyperplane $H\subseteq\RR^{n}$ such that
$0\notin H$ and any convex body $K\subseteq H$. Consider $C_{0}=\inc K$.
Then for $x,y\in K$ we clearly have $(x,y)\not\subseteq C_{0}$,
so $\phi(C_{0})$ is not convex. 
\end{proof}
We summarize the discussion in the following corollary:
\begin{corollary}
The following are equivalent for a convex body $K$ with $0\notin K$:
\begin{enumerate}
\item $\phi(K)$ is convex.
\item $\phi\left(\inc K\right)$ is convex. 
\item For every $x,y\in\inc K$ we have $(x,y)\in\inc K$. 
\end{enumerate}
\end{corollary}

\begin{proof}
By Lemma \ref{lem:in-out} the set $\phi(K)$ is convex if and only
if $\phi\left(\outc K\right)=\inc\left(\phi(K)\right)$ and $\phi\left(\inc K\right)=\outc\left(\phi(K)\right)$
are both convex. By the theorem we know that $\phi(\outc K)$ is always
convex, so the equivalence of (1) and (2) follows. The equivalence
of (2) and (3) follows immediately from Proposition \ref{prop:arc-inversion}. 
\end{proof}

\section{Local Theory of Flowers}

The main goal of this section is to prove that every origin-symmetric
flower $F\subseteq\RR^{n}$ has a ``large'' dimensional projection
which is ``almost'' a Euclidean ball. More formally, we define the
geometric distance of two origin-symmetric star bodies $A,B\subseteq\RR^{n}$
to be 
\[
d(A,B)=\inf\left\{ a\cdot b:\ \frac{1}{a}A\subseteq B\subseteq b\cdot A\right\} .
\]
Note that for every convex body $K$ we have 
\[
d(K,B_{2}^{n})=\frac{\max_{\theta\in S^{n-1}}h_{K}(\theta)}{\min_{\theta\in S^{n-1}}h_{K}(\theta)}=\frac{\max_{\theta\in S^{n-1}}r_{K^{\flower}}(\theta)}{\min_{\theta\in S^{n-1}}r_{K^{\flower}}(\theta)}=d(K^{\flower},B).
\]

Our theorem is then as follows:
\begin{theorem}
\label{thm:flower-dvoretzky}Let $F\subseteq\RR^{n}$ be a flower
with $F=-F$. Then for every $\epsilon>0$ there exists a subspace
$E$ of dimension $\dim E=c(\epsilon)\cdot n$ such that 
\[
d\left(P_{E}F,B_{2}^{E}\right)\le1+\epsilon.
\]
\end{theorem}

Here of course $B_{2}^{E}=B_{2}^{n}\cap E$. The same theorem for
convex bodies instead of flowers is of course the famous Dvoretzky's
theorem. However, in Dvoretzky's theorem the dependence of $\dim E$
on $n$ is much poorer: It is well known that if for example 
\[
K=B_{1}^{n}=\conv\left\{ \pm e_{1},\pm e_{2},\ldots,\pm e_{n}\right\} \subseteq\RR^{n}
\]
 and $d\left(P_{E}K,B_{2}^{E}\right)\le1+\epsilon$, then necessarily
$\dim E\le c(\epsilon)\log n$. This is in contrast with our theorem,
where $\dim E$ is proportional to $n$. 

On a related note, observe that if $F=\left(B_{1}^{n}\right)^{\flower}$
then for every subspace $E$ we have 
\[
d\left(F\cap E,B_{2}^{E}\right)=d\left(\left(P_{E}B_{1}^{n}\right)^{\flower},B_{2}^{n}\right)=d\left(P_{E}B_{1}^{n},B_{2}^{n}\right),
\]
 which is large unless $\dim E\le c(\epsilon)\log n$. In other words,
there is no analogue of Theorem \ref{thm:flower-dvoretzky} where
the projection $P_{E}F$ is replaced with the intersection $F\cap E$.
This is unlike the classical theorem for convex bodies, where the
two versions are easily seen to be equivalent. 

Of course, our theorem does imply a corollary for convex bodies:
\begin{corollary}
Let $K\subseteq\RR^{n}$ be a convex body with $K=-K$. Then for every
$\epsilon>0$ there exists a subspace $E$ of dimension $\dim E=c(\epsilon)\cdot n$
such that 
\[
d\left(\left(P_{E}K^{\flower}\right)^{-\flower},B_{2}^{E}\right)\le1+\epsilon.
\]
 
\end{corollary}

However, as was explained in the introduction, we do not have a good
understanding of the body $\left(P_{E}K^{\flower}\right)^{-\flower}$. 

For the proof of Theorem \ref{thm:flower-dvoretzky} we will need
the following result, which may be of independent interest:
\begin{theorem}
\label{thm:flower-stability}Let $F\subseteq\RR^{n}$ be a flower.
Assume that $d(\conv F,B_{2}^{n})=1+\epsilon$ for $\epsilon<\frac{1}{10}$.
Then 
\[
d(F,B_{2}^{n})\le1+3\sqrt{\epsilon}.
\]
\end{theorem}

Obviously, the assumption that $F$ is a flower is crucial, as no
such theorem can be true for general star bodies. 
\begin{proof}
We may assume without loss of generality that $\frac{1}{1+\epsilon}B_{2}^{n}\subseteq\conv F\subseteq B_{2}^{n}$.
We obviously have $F\subseteq B_{2}^{n}$.

Fix a direction $\theta\in S^{n-1}$ and write $w=\frac{1}{1+\epsilon}\theta\in\conv F$.
Consider the half-space 
\[
H=\left\{ z:\ \left\langle z,\theta\right\rangle \ge\frac{1}{1+\epsilon}\right\} .
\]
 We claim that $H\cap F\ne\emptyset$: If not then $F\subseteq H^{c}$,
and then $\conv F\subseteq H^{c}$ so $w\in H^{c}$. This is a contradiction.
We may therefore choose a point $x\in H\cap F$. Since $F\subseteq B_{2}^{n}$
we obviously have $\left|x\right|\le1$.

Since $x\in F$ we know that $x\in B_{y}\subseteq F$ for some $y\in\RR^{n}$,
where as usual 
\[
B_{y}=B\left(\frac{y}{2},\frac{\left|y\right|}{2}\right)=\left\{ z:\ \left|z\right|^{2}\le\left\langle z,y\right\rangle \right\} .
\]
 Since $y\in B_{y}\subseteq F\subseteq B_{2}^{n}$ we have $\left|y\right|\le1$. 

Our goal is to understand for which values of $\lambda$ we have $\lambda\theta\in B_{y}$.
Towards this goal observe that $x\in H$ and hence 
\[
\left\langle x,w\right\rangle =\frac{1}{1+\epsilon}\left\langle x,\theta\right\rangle \ge\frac{1}{\left(1+\epsilon\right)^{2}},
\]
 so 
\[
\left|x-w\right|^{2}=\left|x\right|^{2}-2\left\langle x,w\right\rangle +\left|w\right|^{2}\le\left|x\right|^{2}-\frac{2}{(1+\epsilon)^{2}}+\frac{1}{(1+\epsilon)^{2}}=\left|x\right|^{2}-\frac{1}{(1+\epsilon)^{2}}.
\]
 It follows that 
\begin{align*}
\left\langle w,y\right\rangle  & =\left\langle x,y\right\rangle -\left\langle x-w,y\right\rangle \ge\left|x\right|^{2}-\left|x-w\right|\left|y\right|\\
 & \ge\left|x\right|^{2}-\left|x-w\right|\ge\left|x\right|^{2}-\sqrt{\left|x\right|^{2}-\frac{1}{(1+\epsilon)^{2}}}
\end{align*}
 where we used the fact that $x\in B_{y}$ and Cauchy-Schwarz. 

Simple calculus shows that if $\frac{3}{4}<a<1$ then the function
$\phi(t)=t-\sqrt{t-a}$ is decreasing on the interval $[a,1]$. For
$\epsilon<\frac{1}{10}$ we have $\frac{1}{(1+\epsilon)^{2}}>\frac{3}{4}$,
so
\[
\left\langle w,y\right\rangle \ge\left|x\right|^{2}-\sqrt{\left|x\right|^{2}-\frac{1}{(1+\epsilon)^{2}}}\ge1-\sqrt{1-\frac{1}{(1+\epsilon)^{2}}.}
\]
Therefore 
\[
\left\langle \theta,y\right\rangle =(1+\epsilon)\left\langle w,y\right\rangle \ge1+\epsilon-\sqrt{\left(1+\epsilon\right)^{2}-1}=1+\epsilon-\sqrt{\epsilon^{2}+2\epsilon}\ge1-\sqrt{2\epsilon}.
\]
 It follows that if $\lambda\le1-\sqrt{2\epsilon}$ then $\lambda\le\left\langle \theta,y\right\rangle $,
so 
\[
\left|\lambda\theta\right|^{2}=\lambda^{2}\le\lambda\cdot\left\langle \theta,y\right\rangle =\left\langle \lambda\theta,y\right\rangle ,
\]
 so $\lambda\theta\in B_{y}\subseteq F$. 

Since the direction $\theta\in S^{n-1}$ was arbitrary it follows
that $F\supseteq\left(1-\sqrt{2\epsilon}\right)\cdot B_{2}^{n}$,
so 
\[
d(F,B_{2}^{n})\le\frac{1}{1-\sqrt{2\epsilon}}\le1+3\sqrt{\epsilon},
\]
 where we used again that $\epsilon<\frac{1}{10}$. 
\end{proof}
Now we can prove Theorem \ref{thm:flower-dvoretzky}:
\begin{proof}[Proof of Theorem \ref{thm:flower-dvoretzky}.]
We may of course assume that $\epsilon<\frac{1}{2}$. It was proved
in \cite{Milman2019} that the set $\widetilde{F}=\conv F$ is also
a flower, and that for convex flowers we always have $d\left(\widetilde{F},B_{2}^{n}\right)\le2$. 

$\widetilde{F}\subseteq\RR^{n}$ is an origin-symmetric convex body
which is a bounded distance from $B_{2}^{n}$. From Milman's version
of of Dvoretzky's theorem it follows that there exists $E\subseteq\RR^{n}$
of dimension $c(\epsilon)\cdot n$ such that 
\[
d\left(\conv\left(P_{E}F\right),B_{2}^{E}\right)=d\left(P_{E}\widetilde{F},B_{2}^{E}\right)\le1+\epsilon^{2}/9
\]
(see \cite{Milman1972}, or for example \cite{Artstein-Avidan2015})

Since $\epsilon^{2}/9<\frac{1}{10}$ and since $P_{E}F$ is a flower
we may apply Theorem \ref{thm:flower-stability} and deduce that 
\[
d\left(P_{E}F,B_{2}^{E}\right)\le1+3\sqrt{\epsilon^{2}/9}=1+\epsilon.
\]
\end{proof}

Recall the easy observation from the start of this section that 
\[
d(K,B_{2}^{n})=d\left(K^{\flower},B_{2}^{n}\right).
\]
 From here, many results for flowers follow immediately from the corresponding
results for convex bodies. We demonstrate below one such result:
\begin{theorem}
\label{thm:global-dvoretzky}There is a universal constant $C>0$
such that for every dimension $n,$ every flower $F\subseteq\RR^{n}$
and every $\epsilon>0$ there exists $N\le C\cdot\frac{n}{\epsilon^{2}}$
and $N$ rotations $\left\{ u_{i}\right\} _{i=1}^{N}$ such that for
some $r>0$ and for all $\theta\in S^{n-1}$ we have 
\[
(1-\epsilon)r\le\frac{1}{N}\sum_{i=1}^{N}r_{u_{i}F}(\theta)\le(1+\epsilon)r.
\]
\end{theorem}

(Here $r_{u_{i}F}$ is the radial function of the flower $u_{i}F$).

The statement of this theorem is highly non-trivial, but it is just
an interpretation of known theorems from Asymptotic Geometric Analysis.
This is Theorem 2 from \cite{Bourgain1988} (and one may also see
it in \cite{Bourgain1989a}, Theorem 6.3). However, the estimate in
these papers have an extra factor $\log\frac{1}{\epsilon}$, which
was later eliminated by Schmuckenschlager in \cite{Schmuckenschlager1991}.

An interesting particular case is the case of $F=B_{\theta}$, i.e
a single petal, which is the flower of $I_{\theta}=[0,\theta]^{\flower}$.
In this particular case the result for intervals $\left\{ I_{\theta_{i}}\right\} _{i=1}^{N}$
was known much earlier. It follows from \cite{Figiel1976}, but in
the dual presentation (see also Proposition 6.2 in \cite{Bourgain1989a}
for the result with equal length intervals). 

Actually, in this particular case one may choose just $2n$ petals
and receive an isomorphic version of Theorem \ref{thm:global-dvoretzky}.
In other words, there are two universal constants $c_{2}>c_{1}>0$
such that 
\[
c_{1}r\le\frac{1}{2n}\sum_{i=1}^{2n}r_{u_{i}F}(\theta)\le c_{2}r.
\]
 This is already a consequence of Kashin's theorem from \cite{Kashin1977}.
 




\subsection*{Supports.} The first author is partially supported by the ISF grant 519/17 and the second author is partially supported by ISF grant 1468/19. Both authors are jointly supported by BSF grant 1468/19. 


\bibliography{paper.bbl}


\EndPaper


\end{document}